\documentclass[10pt,twocolumn,twoside]{IEEEtran}
\IEEEoverridecommandlockouts
\usepackage{cite}
\usepackage{amsmath,amssymb,amsfonts,mathtools}
\usepackage{amsthm}
\usepackage{algorithmic}
\usepackage{graphicx}
\usepackage{textcomp}
\usepackage{xcolor}
\usepackage[linesnumbered,ruled,vlined]{algorithm2e}
\usepackage[shortlabels]{enumitem}
\usepackage{bm}
\usepackage{tikz}
\usepackage{stfloats}
\usetikzlibrary{plotmarks}
\newcommand\marksymbol[2]{\tikz[#2,scale=1.2]\pgfuseplotmark{#1};}
\usetikzlibrary{decorations.pathreplacing}
\usetikzlibrary{circuits.ee.IEC}
\def\BibTeX{{\rm B\kern-.05em{\sc i\kern-.025em b}\kern-.08em
    T\kern-.1667em\lower.7ex\hbox{E}\kern-.125emX}}
\newtheorem{theorem} {Theorem}
\newtheorem{example} {Example}
\newtheorem{definition} {Definition}
\newtheorem{assumption} {Assumption}
\newtheorem{lemma} {Lemma}
\newtheorem{proposition}{Proposition}
 \newtheorem*{remark}{Remark}

\newcounter{subeqn} %

\newcommand{\ts}[1]{{\textnormal{#1}}}

\newcommand{\ie}{\emph{i.e.},\ }

\newcommand{\Nset}{\mathbb{N}}
\newcommand{\Rset}{\mathbb{R}}

\newcommand{\mc}{\mathcal}
\newcommand{\mb}{\mathbf}
\newcommand{\mbb}{\mathbb}

\DeclareFontFamily{U}{mathx}{\hyphenchar\font45}
\DeclareFontShape{U}{mathx}{m}{n}{
      <5> <6> <7> <8> <9> <10>
      <10.95> <12> <14.4> <17.28> <20.74> <24.88>
      mathx10
      }{}
\DeclareSymbolFont{mathx}{U}{mathx}{m}{n}
\DeclareFontSubstitution{U}{mathx}{m}{n}
\DeclareMathSymbol{\bigtimes}{1}{mathx}{"91}

\allowdisplaybreaks
\begin{document}

\title{A Coalitional Game for Demand-Side Management in a Low-Voltage Resistive Micro-Grid with Multiple Electricity Retailers
\thanks{*Sponsored by Mexico's CONACyT, scholarship number: 440742.}
\thanks{This work was supported by EPSRC under Grants No EP/S001107/1 and EP/S031863/1.}
\thanks{F. Genis Mendoza, and P. R. Baldivieso-Monasterios are with the Department of Automatic Control and Systems Engineering, The University of Sheffield, Mappin Street, S1 3JD Sheffield, United Kingdom. 
{\tt\small \{fgenismendoza1, p.baldivieso\}@sheffield.ac.uk}}%
\thanks{G Konstantopoulos is with the Department of Electrical and Computer Engineering, University of Patras, Rion, 26500, Greece.
  {\tt\small{g.konstantopoulos@ece.upatras.gr}}}
\thanks{D. Bauso is with the Jan C. Willems Center for Systems and Control, ENTEG, Fac. Science and Engineering, University of Groningen, Nijenborgh 4, 9747 AG Groningen, The Netherlands and Dipartimento di Ingegneria, Universit\`a di Palermo, viale delle Scienze, Palermo, Italy
{\tt\small d.bauso@rug.nl}}%
}

\author{Fernando Genis Mendoza, Pablo R. Baldivieso-Monasterios*, Dario Bauso and George Konstantopoulos}

\maketitle

\begin{abstract}
  An existing challenge in power systems is the implementation of optimal demand management through dynamic pricing. This paper encompasses the design, analysis and implementation of a novel on-line pricing scheme based on coalitional game theory. The setting consists of a network with multiple energy retailers competing to attract consumers by announcing a price in a hierarchical leader-follower structure. The process of coalition formation under such a pricing scheme can be viewed as a game for which we show that a Stackelberg equilibrium exists, \ie given a price, consumers will respond by conforming to a reciprocal power consumption quantity. We propose a coalition formation algorithm and perform a game-theoretic stability analysis on the resulting coalitions. We integrate the pricing setting with a resistive micro-grid dynamic model. In this context we analyse the behaviour of the integrated system, bridging the gap between market and physical layers of the problem. Simulations provide a comparison of profits generated by the proposed scheme against a more traditional single retailer scheme, while simultaneously showing convergence towards a steady-state equilibrium. Additionally, we shed light into the system's physical response when subject to our proposed pricing scheme.
\end{abstract}


\section{Introduction}
With the advent of the smart grid paradigm, more ways to distribute power have emerged, bringing changes to the electricity market. Here, governments and general consumers now seek and switch to better providers and sources of power that enable them to fulfill certain requirements while still being profitable. Such paradigm has also enabled a higher degree of communication between consumers and energy retailers, providing new set-ups where these interact and cooperate to optimise their outputs. Dynamic pricing schemes in electrical systems establish a viable way to optimise and shift consumption during peak times without compromising generation and distribution systems \cite{Namerikawa}. These schemes allow retailers to make tariff changes based on demand whilst helping users minimise costs of their consumption \cite{Loni2017}. An underlying assumption is that both suppliers and consumers are rational, namely they aim to optimise their behaviour, and as consequence, a price change entails a change in consumption \cite{Namerikawa}. 



The landscape of dynamic pricing algorithms in the literature is vast. In this brief literature review, we focus on describing pricing mechanisms in game-theoretic terms. The framework of cooperative games, as mentioned in \cite{Loni2017}, provides a natural method for describing and modeling pricing in electric networks. The use and formulation of cooperative coalitional games in micro-grids was first introduced in \cite{Saad2011}, where the proposed algorithm focuses on reducing power losses and costs by forming coalitions of neighboring micro-grids. A game where local micro-grids cooperate without participation of a main grid is presented in \cite{Chis2019} by assuming a known non-flexible demand. In \cite{Mei2019} the introduction of auction theory is used to define the pairing of micro-grids that buy and sell. In \cite{Wei2014}, the authors propose forming coalitions in a macro station, and using the Shapley value to distribute profits. A similar approach is used in \cite{Pilling2017}, where a case study is performed with data of an existing micro-grid and a utility grid with fictitious interconnections. In \cite{Xu2014} and \cite{Ghiasvand2022}, we can see, from a numerical point of view, the relation between coalition formation and changes in energy prices. A study of the case where greedy prosumers do not align with the micro-grid's decision is presented in \cite{Han2019}, here a balanced game is proposed without the need of calculating the imputation set. A notion of fairness is introduced by \cite{Moafi2023} using particle swarm optimisation and using the nucleolus as a game solution concept. A bidding system for cooperating prosumers is presented in \cite{Chakraborty2019}, however constraints on power capacity and losses are ignored. The use of evolutionary game theory in conjunction with coalitions is proposed in \cite{Mondal2017}, where the price is a quadratic function of the consumption. In summary, coalitional game theory represents a powerful modelling tool to understand and enhance pricing behaviours within an electric network.

We have, however, identified two fronts on which the existing body of literature has had limited contributions. Existing approaches often neglect the effect of pricing on the physical network. Additionally, with a few exceptions \cite{Chis2019,BaldiviesoMonasterios2022,Tarashandeh2024}, the coalitional games proposed in the literature do not consider end consumers as players. In this paper, we aim to tackle both of these challenges. The pricing scheme is modelled as a Stackelberg-based game with incentive strategies where both \emph{retailers} and \emph{consumers} are players. Retailers adjust their prices periodically, which in consequence will cause them to lose or gain consumers, and at the same time, adjust the demand and supply of energy in the network. We build upon our previous results in \cite{GenisMendoza2021} to analyse the game equilibria and coalition stability properties. We employ the conductance matrix of the network as a tool in a cost-saving coalitional game; this, to the best of the author's knowledge, has not been done before. From the physical perspective, we analyse the interactions between dynamic pricing and a modified droop based control laws governing the voltage dynamics. We leverage on the bounded integral control theory from \cite{Konstantopoulos2019a}. Our approach aims to provide a rigorous framework for modern energy trading platforms such as \cite{GoodEnergyLTD2020} and \cite{UswitchLimited}.

The contributions and structure of this paper are:
\begin{itemize}
\item We propose a coalitional game framework in Section~\ref{p4sec:game}, where there are multiple competing retailers in a micro-grid. To the best of the authors' knowledge, this is a novel problem setup in micro-grid literature. We analyse the stability, in a game-theoretic sense, of the coalitions generated by our algorithm. Additionally, we introduce the basis for a novel method where a cost-saving game is linked and formulated directly from the network conductance matrix.
\item In Section~\ref{p4sec:stats}, we show how consumers can improve their welfare by sharing risks. The analysis further stresses the benefits of the coalitional approach in a pricing setting.
\item We derive in Section~\ref{p4sec:alltogether} stability conditions, from a dynamical systems perspective, for the physical system subject to our pricing algorithm. To the best of the author's knowledge, this is the first instance of such combined analysis.
\end{itemize}
%
%
\textbf{\emph{Notation and  definitions}}: A \emph{game} $(\mc{N},\{\mc{H}_i\}_{i\in\mc{N}},\{u_i\}_{i\in\mc{N}})$ is a triplet composed by a set of players $\mc{N}$, a collection action sets $\mc{H}_i$, and a collection of payoff functions $u_i\colon\mc{H}_i\to\Rset$. An \emph{action profile} for a set of $|\mc{N}|$ players is an ordered tuple $(h_1,\ldots,h_{|\mc{N}|})$. The \emph{best response set} of player $i\in\mc{N}$ is defined as $\mc{Q}_i(h_{-i}):= \arg\max\{ u_i(h_i,h_{-i})\colon h_i \in \mathcal{H}_i\}$ where $h_{-i} = (h_j)_{j\in\mc{N}\setminus\{i\}}$. For $\mc{N} = \{1,2\}$, an \emph{action profile} $(h_1^S,h_2^S)$ is a \emph{Stackelberg Equilibrium} for player 1 if $h_2^S\in\mathcal{Q}_2(h_1^S)$ and $u_1(h_1^S,h_2^S)\geq u_1(h_1,h_2)$, $\forall h_1\in\mathcal{H}_1,h_2\in\mathcal{Q}_2(h_1)$. A coalitional game is defined by the pair $(\mathcal{N},\nu)$ with $\mc{N}$ a set of players and $\nu\colon 2^{\mathcal{N}}\to\Rset$ its characteristic function. A pay-off vector is an element $a\in\Rset^{|\mc{N}|}$ such that $a_i\in\Rset$ represents the gains of player $i\in\mc N$. The \emph{core}, $\mc{C}\subset\Rset^{|\mc{N}|}$, of a coalitional game $(\mc{N},\nu)$ satisfies $\mc{C} = \{a\in\Rset^{|\mc{N}|}\colon \sum_{i\in S}a_i\geq\nu(S),~\forall S\subset\mc{N}\}$. A cooperative game $(\mc{N},\nu)$ with $\nu(\emptyset)= 0$ is: a \emph{convex game} if $\nu(S\cup T)+\nu(S\cap T)\leq \nu(S)+\nu(T)$ $\forall S,T\subseteq \mc{N}$; a \emph{permutationally convex} (PC)  game if there exists a permutation $\pi$ of $\mc N$ such that $\forall S\subseteq \mc{N}\setminus [k]$ and $k> j$, $\nu([k]\cup S)-\nu([k])\leq \nu([j]\cup S)-\nu([j])$ where $[k] = \{k\}\cup\{i\in\mc{N}\colon \pi(i)<\pi(k)\}$.

A power network is defined by a connected, undirected, and weighted graph $\mc{G} = \mc{(N,E)}$, where $\mc{N}$ is the set of nodes, and $\mc{E}\subseteq\mc{N}\times\mc{N}$ the set of edges defining the interconnection topology. A $k-$path from node $i$ to $j$ is a sequence of $k$ edges $\phi(i,j) = \{e_1,\ldots,e_k\}\subset\mc{E}$. The mapping $\omega\colon\mc{E}\to\Rset$ defines the weight of each edge such that  $\omega(i,j) = \omega_{ij}\in\Rset$. For undirected graphs, the mapping $\omega(\cdot,\cdot)$ is symmetric and can be characterised by a symmetric adjacency matrix $A\in\Rset^{|\mc{N}|\times|\mc{N}|}$. The matrix $B\in\Rset^{|\mc{N}|\times|\mc{N}|}$, such that $B_{ij} = 1$ if $A_{ij}>0$ and $B_{ij} = 0$ otherwise, and $A$ determine the connectivity properties of $\mc{G}$. From \cite{bullo}, $B^k$ and $A^k$ with $k\in\Nset$ determine the number of $k-$paths and the sum of products of weights of all $k-$paths respectively from nodes $i$ to $j$. The total weight of $\mathcal{G}$ is $\omega(\mathcal{E})=\sum_{e\in\mathcal{E}} \omega_e = \frac{1}{2}\sum_i\sum_jA_{ij}$. The out degree and Laplacian matrices of $\mc{G}$ are $D\coloneqq \ts{diag}(\sum_{j=1}^N A_{ij})$, and $\mc{L} = D-A$ respectively. The neighbours of node $i$ is $\mc{N}_i = \{j\in\mc{N}\colon\exists (i,j)\in\mc{E}\}$; the out-degree of node $i$ is $\delta_i = |\mc{N}_i|$. A tree is an undirected graph in which any two vertices are connected by a sequence of edges. The \emph{Minimum Spanning Tree} (MST) of $\mathcal{G(N,E)}$ is a tree $T=(\mathcal{N, E^*})$ such that ${\mc{E}^*} = {\arg\min}\{\omega(\mc{E}')\colon \mc{E}'\subset\mc{E},\forall i\in\mc{N},\exists j\in\mc{N}, (i,j)\in\mc{E}'\}$.
\section{System Model, Definition and Preliminaries} \label{p4sec:models}
In this section, we make precise the notions of retailers and consumers. In our demand-side management problem, electricity prices regulate user consumption. The consumer response to announced prices involves a hierarchical structure fitting within a Stackelberg game framework \cite{Bausoa}. In this setting, \emph{leaders}, (retailers) play first by selecting a price; and a subset of \emph{followers} (consumers) select their power consumption at a given price. When consumers and retailers cannot obtain more benefit by modifying consumption or price, then the game is at an equilibrium.

%
%
%
 \subsection{Coalitional setting} \label{p4sec:setscoalitions}
 To study the coalitional behaviour of our scheme, we employ a game-theoretic framework. The set of players in the micro-grid is $\mc{N} = \{1,\ldots,N\}$ and involves $N$ players.  This is partitioned into two non-overlapping sets: the set of retailers $\mathcal{R}\subset\mathcal{N}$ and the set of consumers $\mathcal{B}\subset\mathcal{N}$ with  $\mathcal{R}\cup\mathcal{B}=\mathcal{N}$ and $\mc{R}\cap\mc{B}=\emptyset$. Since $\mc{R}$ and $\mc{B}$ form a partition of $\mc{N}$, then $N = |\mc{B}| + |\mc{R}|$ where $|\mc{B}|$ is the number of consumers, and $|\mc{R}|$ is the number of retailers. Besides this basic partition, we seek a pairing between a retailer and a subset of consumers.
 \begin{definition}[Retailer's coalition] \label{p4def:coal}
   For $r\in\mc{R}$, the \emph{retailer coalition} is $S_r\coloneqq\{r\} \cup \mc{B}_r$ where $\mc{B}_r = \{b_1,\ldots,b_k\}\subset\mc{B}$.
\end{definition}
The above definition assigns to each element in the retailer set, a corresponding subset of consumers; this definition considers the case where a retailer $r\in\mc{R}$ does not succeed in attracting any consumers, \ie $\mc{B}_r = \emptyset$ and $S_r = \{r\}$. We invoke the following assumptions:
\begin{assumption}\label{p4ass:union}
  The collection of sets $\{S_r\}_{r\in\mc{R}}$ forms a covering of $\mc{N}$, \ie $\bigcup_r S_r=\mathcal{N}$.
\end{assumption}
A consequence of Assumption~\ref{p4ass:union} is that each consumer $b\in\mc{B}$ needs to be assigned to one retailer. Coalitions that have more than one retailer or that share consumers are considered nonviable, this is formalised as follows.
\begin{assumption}\label{p4ass:coalcond}[Viable coalitions]
  \begin{enumerate}
  \item\label{p4eq:nosellers} For $r\in\mc{R}$, its associated retailer coalition $S_r\subset\mc{N}$ satisfies $\forall s\in\mc{R}\setminus\{r\}$, $S_r\cap \mc{R}\setminus\{s\} = \emptyset$.
  \item\label{p4eq:nooverlap} For any $r,s\in\mc{R}$, $S_r\cap S_s = \emptyset$.
  \end{enumerate}
\end{assumption}
Coalitions satisfying Assumptions~\ref{p4ass:union} and \ref{p4ass:coalcond} are considered viable coalitions and exclude all those containing more than one or no retailer.
\subsection{Consumer and Retailer Profit Functions}
\label{p4sec:players}
In our problem setting, both consumers and retailers are considered to be price-taking rational agents, \ie both aim to maximise their profit for producing or consuming energy. The profit function for a retailer $r\in\mc{R}$ is $ \Pi _{r} (S_r,\lambda_r,P)= {\lambda_r} P - C_r((1+\alpha_r^{loss})P,S_r),$  where $C_r(\cdot,\cdot)$ is a function corresponding to the cost of producing $(1+\alpha_r^{loss})P$ units of power distributed among consumers in $S_r$ at a price $\lambda_r$ and a network dependent loss coefficient $\alpha_r^{loss}\in[0,1]$. Similarly, every consumer $b$ that has opted to consume from~$r$ calculates its profit as $\Pi _{b}(\zeta,r)= U_b(\zeta,r) - {\lambda_r}\zeta$ where $U_b(\cdot,\cdot)$ is the utility from consuming $\zeta$ amount of power. Following~\cite{Namerikawa}, we impose restrictions on the choice of functions:
 \begin{assumption}
   For each $(r,b)\in\mc{R}$ and $b\in\mc{B}$, the function $C_r(\cdot,S_r)\colon\Rset\to\Rset$  ($U_b(\cdot,r)\colon\Rset\to\Rset$) is continuous, monotonically increasing, and convex (concave).
   \label{assum:costs}
 \end{assumption}
 Each retailer in our setting seeks to maximise its profits by selling power to a subset of consumers $\mc{B}_r$ at a price ${\lambda_r}$. Similarly, each consumer seeks to maximise its benefit by consuming ${P^{d}_{{b}}}$ power from a retailer $r\in\mc{R}$. This process is captured by the following coupled optimisation problems: 
 \begin{subequations}
   \begin{align}
     \mbb{P}_r(P^d_{\mc{S}_r}) \colon &\mathop {\max }\limits_{{B}\subset\mc B,~\lambda  \in [\underline \lambda  ,\bar \lambda ]}~ ~\Pi _{r} (\{r\}\cup B,\lambda,\sum_{b\in {B}} P_{{b}}^d),\label{p4eq:optprice_mi}\\
     \mbb{P}_b(r) \colon &\mathop {\max }\limits_{\zeta  \in [\underline \zeta  ,\bar \zeta ]}~\Pi_{b}(\zeta,r),\label{p4eq:optcons_mi}
   \end{align}\label{p4eq:opt_mi}
 \end{subequations}
\noindent where $P^d_{\mc{S}_r} = \{P_b^d\}_{b\in\mc{S}_r}$ is the collection of power demands in the coalition of $r$ and $[\underline \lambda  ,\bar \lambda ]\subset\Rset$ and $[\underline \zeta  ,\bar \zeta ]\subset\Rset$ are the price and power consumption limits respectively. Each retailer $r\in\mc{R}$ cannot provide power to the consumers beyond its own generation capabilities
 \begin{equation}
   (1+\alpha_r^{loss})\biggr|\sum_{b\in S_r} P^d_{{b}}\biggl|\leq P_r^\ts{max}.
   \label{p4eq:powerlimit}
 \end{equation}
The optimisation problems defined in \eqref{p4eq:opt_mi} are of mixed-integer type which makes their use in an online setting problematic. This problem, however, is of interest conceptually because its solution yields both the optimal price for each retailer $\lambda_r^*$, the optimal consumption for each consumer ${P_b^d}^*$, and the optimal partition $\{\mc{B}_r^*\cup\{r\}\}_{r\in\mc R}$ of $\mc N$. 
\section{Coalitional Game with Multiple Retailers} \label{p4sec:game}
In this section, we propose an algorithmic way of handling the coupled optimisation problems $\mbb{P}_r(\cdot)$ and $\mbb{P}_b(\cdot)$ for each $r,b\in\mc{N}$. In fact, Algorithm~\ref{p4alg:alg1} defines the coalition formation procedure that employs game-theoretic concepts to achieve a Stackelberg equilibrium of a cooperative game. We show that the outcome of this coalition formation procedure is in fact a solution to \eqref{p4eq:opt_mi} and is a stable solution of the game. 
\subsection{Cost Definition and Minimum Spanning Tree Problem} \label{p4sec:networks}
For a retailer $r\in\mc{R}$, its associated coalition $S_r\subset\mc{N}$ induces a connection sub-graph $\mc{G}_r = (S_r,\mc{E}_r)$ denoted here as \emph{retailer's cost network}. The set of edges $\mc{E}_r$ comprises all links of the form $(r,b)$ with $b\in\mc{B}_r$; the associated weights $\omega(r,b)$ determine the \emph{direct connection} costs and may represent different factors such as physical position with respect to the supplier, power losses, or fees imposed. The links of the type $(b,d)$ with $b,d\in\mc{B}_r$ and their associated weights represent \emph{aggregate connection} and may represent benefits of joining a coalition. The total value $c\colon2^{\mc{N}}\to\Rset$ of coalition $(S,\mc{E}_S)$ is given by the MST, \ie $c(S) = \omega(\mc{E}^*_S)$.
%
%
%
\begin{figure}[t!]
\centering
\begin{tikzpicture}
\pgfdeclarelayer{bg}    
\pgfsetlayers{bg,main}
\draw[ultra thick] (-1,4.75) circle (.3) node[anchor=center] {$b_3$};
\draw[ultra thick](-0.7,7.25)to (0.7,6.3);
\draw[dashed](-0.7,7.15)to (0.65,6.25);
\draw[ultra thick] (-1,7.25) circle (.3) node[anchor=center] {$b_1$};
\draw[ultra thick] (-3.5,6) circle (.3) node[anchor=center] {$r_1$};
\draw (-2.45,7.05) node[anchor=center] {100};
\draw (0.15,6.95) node[anchor=center] {30};
\draw (-2,6.1) node[anchor=center] {80};
\draw (-0.75,6.65) node[anchor=center] {40};
\draw (-2.55,5.05) node[anchor=center] {90};
\draw[ultra thick] (0.7,6) circle (.3) node[anchor=center] {$b_2$};
\draw[ultra thick](-1.3,7.25)to (-3.5,6.3);
\draw[ultra thick](-1,5.05)to (-1,6.95);
\draw[dashed](-1.1,5.05)to (-1.1,6.95);
\draw[ultra thick](-3.5,5.7)to (-1.3,4.75);
\draw[dashed](-3.2,5.8)to (0.4,5.8);
\draw[ultra thick](0.4,5.9)to (-3.2,5.9);
\begin{pgfonlayer}{bg}
  \draw[red!50!black] [thick ] (3.1,5.95) circle (.3) node[anchor=center] {$r_j$};
  \draw[red!50!black] [thick](-0.8,7.5)to (3.1,6.25);
  \draw[red!50!black] [thick](-0.8,5)to (0.55,5.75);
  \draw[red!50!black] [ thick](-0.8,7.05)to (0.45,6.15);
  \draw[red!50!black] [ thick](3.1,5.65)to (-0.75,4.6);
  \draw[red!50!black] [thick](1,5.95)to (2.8,5.95);
  \draw[red!50!black] (2,6.95) node[anchor=center] {$(r_j,b_1)$};
  \draw[red!50!black] (1.95,6.15) node[anchor=center] {$(r_j,b_2)$};
  \draw[red!50!black] (2.1,5) node[anchor=center] {$(r_j,b_3)$};
  \draw[red!50!black] (0.45,5.25) node[anchor=center] {$(b_2,b_3)$};
  \draw[red!50!black] (-0.35,6.2) node[anchor=center] {$(b_1,b_2)$};
\end{pgfonlayer}
\end{tikzpicture}
\caption{Example cost network from retailer $r_1$, and its MST (dashed). Other retailers $r_j$ also assign similar costs (red).}
\label{p4fig:examplegraph}
\end{figure}
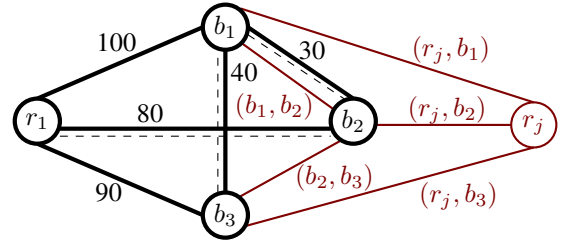
\begin{example}[MST for a coalition ${S}_r$]
  Consider a coalition $S_{r_1}=\{r_1,b_1,b_2,b_3\}$, with a cost network given in Fig.~\ref{p4fig:examplegraph}. This network consists of five different edges $\mc{E}_{r_1}=\{(r_1,b_1),(r_1,b_2),(r_1,b_3),(b_1,b_2),(b_1,b_3)\}$. From the graph, there are seven possible trees that contain all nodes. The MST is found to be the one containing the edges $\{(r_1,b_2),(b_1,b_2),(b_1,b_3)\} $, yielding the cost of the coalition as $c(S_{r_1}) = 150$.
\end{example}
Now we can define the characteristic function of the cooperative game $(\mc{N},\nu)$ defined for all viable coalitions $S_r$:
\begin{equation}
\nu({S_r}) = \sum\limits_{b \in {\mc{B}_r}} \omega( {r},{b})  -  c (S_r).
\label{p4eq:coalvalue}
\end{equation}
%
The value for nonviable subsets, see Assumption~\ref{p4ass:coalcond}, $X\subset\mc{N}$ is  $\nu(X) = 0$. Our goal is to design a cost-saving game which creates an incentive for consumers to join a larger coalition to increase their profits by means of savings. This requirement can be formalised by requiring that savings are larger in larger coalitions.
\begin{assumption}[Coalition savings]
  For a retailer $r\in\mc{R}$, if its associated coalition $S_r$ consists of two or more consumers, then for all $b \in S_r$, $\nu(\{r,b\})\leq \nu(S_r)$.
  \label{assum:savings}
\end {assumption}
In coalitional game theory, the concept of Shapley value is used to define fair imputations \cite{Bausoa}. Consider that consumers enter a coalition $S_r$ in a given order $\sigma = \{r,b_{i_1}\ldots b_{i_k}\}$ where the ordering number of buyer $b_{i_k}$ is denoted as $\sigma^{-1}(b_{i_k})$. For example, if $S_r = \{r,b_3,b_4,b_6\}$ with the ordering $\sigma = (r,b_4,b_3,b_6)$, then $\sigma^{-1}(b_4) = 2$. The set of predecessors of $b$ is $\rho_{b}^\sigma\coloneqq\{s\in S_r\colon \sigma^{-1}(s)<\sigma^{-1}(b)\}$. As a result, the marginal value for $b$ given a sequence $\sigma$ is $m_{{b}}^\sigma = \nu(\rho_{{b}}^\sigma  \cup \{ {b}\} ) - \nu(\rho_{{b}}^\sigma)$ and in vector form  $m^\sigma = (m^\sigma_{b_1},\ldots,m^\sigma_{b_k})$. The \emph{Shapley value} is then calculated as the average of the marginal vector over all permutations of sequences, $\Phi (\nu) = \frac{1}{{k!}}\sum_\sigma  {m^\sigma}$ where $k$ is the total number of consumers in the coalition. The resulting vector outputs the corresponding portion of savings imputed to each consumer $b\in S_r$.
%
\subsection{Retailer's Cost Network Derivation}
A remaining issue is determining the weights of the retailer network $(S_r,\mc{E}_r)$ for $r\in\mc{R}$. For a given $r\in\mc{R}$, we propose the following method to compute the cost network:
\begin{enumerate}
\item \emph{Obtain the adjacency matrices} $A$ and $B$ from the network Laplacian $\mc{L}\in\Rset^{N\times N}$, 
\item The weight $\omega(r,b)$ is the geometric average of all $n-$paths from $r$ to $b$ within the micro-grid $\mc{G}$. Given $\gamma, \xi >0$ and setting $n = 1$,
\[\omega(r,b)=\gamma \biggr(\frac{[A^n]_{rb}}{[B^n]_{rb}}\biggl)^{\frac{1}{n}} + n\xi,\]if $[A^n ]_{rb}= 0$ then set $n = n +1$. Otherwise terminate. 

\item For a pair of consumers $(b_i,b_j)$, the weight is $\omega(b_i,b_j) = \gamma A_{{b_i}{b_j}} +\beta\xi$ if $A_{{b_i}{b_j}}>0$ where $\beta>0$ and $\beta\neq\xi$, and $\omega(b_i,b_j)=0$ otherwise.
\end{enumerate}
The scalars $\gamma$, $\xi$, and $\beta$ represent prices related to edge conductance, direct connection fees, and aggregate connection fees respectively. We note that the cost increases linearly with the number of walks needed to connect two nodes.
\subsection{Price, Consumption, and Coalition Formation}
\label{p4sec:algo}
In this section, we define both cost $C_r(\cdot,\cdot)$ and utility $U_b(\cdot,\cdot)$ function for retailers and consumers respectively. We also establish a link between these functions and those arising from the supplier cost network and how the combination of both leads to a solution of the coalition formation problem. We propose for each $r\in\mc{R}$ and $b\in\mc{B}$ the following cost and utility functions
\begin{subequations}
  \begin{align}
    C(P,S_r)&={\alpha_{{r}}}{(\lambda{P})^2}+\nu(S_r),\label{p4eq:cost}\\
    U(P^d_{b},r)&={ \alpha_{{b}}}{(P^d_{{b}})^{\frac{1}{6}}}+\kappa_{{r}}\delta_{{b}}^{{r}}-\omega(r,b)\label{p4eq:util}
  \end{align}
  \label{eq:cost_util}
\end{subequations}
where $\alpha_r>0$ and $\alpha_b>0$. For the former, the first term is related to the generation costs for each retailer and the second is the savings from the consumers connected to $r\in\mc{R}$ which will be rewarded back via imputation. The utility function employs a concave function similar to that of \cite{GenisMendoza2019,GenisMendoza2022}; a base payment $-\omega(r,b)$ for joining $S_r$; and a \emph{subsidy} ${\kappa_{{r}}\delta_{{b}}^{{r}}}\geq 0$ to consumer $b$ from retailer $r$. The subsidy term is equivalent to the Bahnzaf power index employed in cooperative games dictating how pivotal is a player within a coalition, see \cite{Algaba2019}. A form of the subsidy term is also used in \cite{Namerikawa} as an incentive tool. The potential subsidy is announced by the retailer to the consumer as an incentive for the latter to join its coalition, since in MST games the players with more connections in the cost network can potentially hold more value \cite{Bausoa}, enabling consumers to consequently increase their profits. The subsidy is calculated after the coalition is formed; its value is equal to each consumer's Shapley value $\Phi(\nu)$ imputation. The following lemma provides some properties of both functions.
\begin{lemma}
  For any $r\in\mc{R}$, the functions $C(\cdot,S_r)$ and $U(\cdot,r)$ defined in \eqref{eq:cost_util} satisfy Assumption~\ref{assum:costs}.
  \label{lem:cost_util}
\end{lemma}
%

Our approach aims to solve the optimisations~\eqref{p4eq:opt_mi} in an iterative way. A consumer $b\in\mc{B}$ evaluates prices $\lambda_r$, base payments $\omega(r,b)$, and potential subsidies $\kappa_{r}\delta_{{b}}^{{r}}$ for all $r\in\mc{R}$. Consumer $b$ chooses $r^*$ and then optimises its profit $\Pi_b(\cdot,r^*)$ to yield ${P_b^d}^*$ and reports to $r^*$. Each retailer evaluates submissions from consumers according to its generation constraints \eqref{p4eq:powerlimit}: if the constraint is violated, then retailer $r$ rejects consumers with the least profit in its coalition. When a consumer $b$ is rejected by $r$, then it chooses a different coalition from a retailer contained in the set $\mc{R}\setminus\{r\}$. After all consumers have chosen their retailers, savings are distributed via the Shapley value. This is repeated \emph{ad infinitum}.
%
%
\begin{algorithm}[t!] \label{p4alg:alg1}
  \SetAlgoLined
  \KwData{$\forall r\in\mc{R}$: $\tilde{S}_r$, $P_r^g$, and $\Lambda_r$;  $\forall b\in\mc B$, $P_b^d$}
  \Repeat{$\{S_r\}_{r\in\mc{R}}$ is unchanging}{
  $\forall r\in\mc{R}$, \emph{measure} $\sum_{{b}\in S_r} P_{b}^d$ and \emph{evaluate} \eqref{p4eq:optprice_mi} to obtain $\lambda_r$\;
    $\forall r \in \mc R$, Broadcast $\lambda_r$ and $\kappa_{r}\delta_{{b}}^{{r}}$ to all $b\in\mc{B}$\;
    $\forall b\in\mc B$, solve $\max\{\Pi_b(\zeta,r)\colon \zeta\in[\underline{\zeta},\bar{\zeta}],~r\in\mc R\}$\;
    $\forall b\in\mc B$, announce ${P_b^b}^*$ to $r^*=\arg\max\Pi_b(r,\zeta)$\;
    \Repeat{no consumer $b\in\mc B$ is rejected}{
      \If{for $r\in \mc R$, \eqref{p4eq:powerlimit} does not hold}{
        \Repeat{\eqref{p4eq:powerlimit} holds}{
          $r$ rejects $\hat{b}\in S_r$ with the lowest $P_b^d$\;
        }
        Each rejected $\hat{b}\in\mc{B}$ by $r\in\mc{R}$ solves $\max\{\Pi_{\hat{b}}(\zeta,s)\colon \zeta\in[\underline{\zeta}, P_s^g+\sum_{b\in S_s}P_b^d],~s\in\mc R\setminus\{r\}\}$\;
        Each $\hat{b}\in\mc B$ announces its new coalition $\hat{r}^*$ and demand ${P_{\hat{b}}^d}^*$\;
      }
    }
    }
    $\forall r\in\mc R$, calculate $\nu(S_r)$ and $\Phi(\nu)$\;
    $\forall r\in\mc{R}$, distribute $\Phi(\nu)$ among $b\in S_r\cap\mc B$\;
    $\forall b\in B$, consume agreed demand ${P_b^d}^*$\;
\caption{Coalition Formation}
\end{algorithm}

\subsection{Properties and Stability of the Coalitional Game}
\label{p4sec:mcstvspc}
In this section, we show that coalitions formed using Algorithm~\ref{p4alg:alg1} are stable in a game-theoretic sense. Our proposed algorithm leads to a Stackelberg equilibrium between each retailer and their respective consumers. For $r\in\mc{R}$ and $S_r\subset\mc{M}$, the MST induces a partial order in the elements of $S_r$ with respect to $r$, \ie $b\succ_{r} d$ if $d$ lies in the path connecting $b$ to $r$. We use this order relation to define a PC game. However, we will first investigate existence of a solution for the \emph{minimum cost spanning tree} (MCST) game, namely the existence of a non-empty core given a retailer's coalition $S_r$. Non-emptiness of the core is a known property of convex games; moreover, \cite{Granot1982} has shown that MCST and convex games are PC games and all PC games possess a non-empty core. We are now ready to present the following results:
%
\begin{theorem}
  Suppose Assumptions~\ref{p4ass:union} and \ref{p4ass:coalcond} hold. For $r\in\mc{R}$, the  MCST $( S_r,c)$ is a PC  game with $c\colon 2^{S_r}\to\Rset$ an MST cost.
\label{thm:retailer_PC}
\end{theorem}
\begin{proof}
  For a consumer $b\in\mc{B}\cap S_r$, and $[b]\subset S_r$ an MST ordered set with $b\succ_r b_j$ where $b_j\in[b]$ satisfies $c([b]) > c([b_j])$ and $[b_j]\subset [b]$. Now, given a finite  $T\subset S_r\setminus [b]$, there exist a path $\phi(v_0,v_f)$ in $(S_r,\mc{E}_r)$ such that its end points $v_0,v_f\in S_r$ satisfy $v_0\in T$ and $v_f\in [b]$. The minimum spanning tree for $T\cup[b]$ contains all the edges from the path $\phi(r,b)$ and all those edges $\mc{E}_{b}^T = \bigcup_{t\in T}\bigcup_{b\in[b]}\phi(b,t)$, \ie $c(T\cup[b]) = \omega(\mc{E}_b^T) + c([b])$.  There are two possible cases: the set $T$ contains elements $d\succ_r b$ or $d\not\succ_r b$. For the former, the increment $c(T\cup [b]) - c([b]) = \omega(\mc{E}_b^T)$ which implies\[\begin{split}c(T\cup [b]) - c([b]) & = \omega(\mc{E}_b^T)+c([b_k]) - c([b_k])\\ c(T\cup [b]) - c([b]) & = c(T\cup[b_k]) - c([b_k])\end{split}\]for all $b_k\succ_r b$. For the second case, the set $T\subset S_r\setminus [b]$ contains at least a player that dominates $b\in\mc{B}\cap S_r$. The set of edges $\mc{E}_b^T$ contains two components: the path $\phi(d,b)$ and the remaining edges $\mc{E}_b^T\setminus\phi(d,b)$. Clearly, the increment \[c(T\cup [b]) - c([b]) = \omega(\phi(d,b))+ \omega(\mc{E}_b^T\setminus\phi(d,b)).\]On the other hand, for $b\succ_r b_k$,\[c(T\cup [b_k]) - c([b_k]) = \omega(\phi(d,b_k)) + \omega(\mc{E}_{b_k}^T\setminus\phi(d,b_k)).\]Since $b_k\succ_r d$, then $\omega(\phi(d,b_k)) >\omega(\phi(d,b))$ and $\omega(\mc{E}_{b_k}^T\setminus\phi(d,b_k)) = \omega(\mc{E}_b^T\setminus\phi(d,b))$. As a result,\[c(T\cup [b]) - c([b])<c(T\cup [b_k]) - c([b_k])\]for all $T\subset S_r\setminus[b]$. The MCST $(S_r,c)$ is a PC game.
\end{proof}
\begin{theorem}
Suppose Assumptions \ref{p4ass:union} and \ref{p4ass:coalcond} hold. For $r\in\mc{R}$, The MCST $( S_r,c)$ has a non-empty core.
\end{theorem}
\begin{proof}
The proof is an adaptation of \cite[Theorem 1]{Granot1982} to our setting.
\end{proof}
%
%
With the non-emptiness of the core for cost networks secured, we turn our attention to the game involving multiple retailers. Our objective is to show that competition among retailers results in greater profits. To this aim, the next result establishes subadditivity of the game.
%
\begin{theorem}
  Suppose Assumption~\ref{p4ass:coalcond} holds. The coalitional game with multiple energy retailers $(\mc{N},\nu)$ with value function $\nu(\cdot)$ defined in~\eqref{p4eq:coalvalue} is subadditive, \ie for $r,s\in\mc{R}$
  \begin{equation}\label{p4eq:subadd}
    v(S_r \cup S_s)\leq v(S_s)+v(S_r).
  \end{equation}
\label{p4thm:comps}
\end{theorem}
\vspace{-8mm}
\begin{proof}
  For $r,s\in\mc{R}$ and from Assumption~\ref{p4ass:coalcond}, $\nu(S_r\cup S_s)=0$. On the other hand, $\nu(S_r) \geq 0$ since~\eqref{p4eq:coalvalue} depends on the sum of direct connection costs $\omega(r,b)$ with $b\in\mc{B}\cap S_r$ which is at worst equal to the MST cost $c(S_r)$. Therefore \eqref{p4eq:subadd} holds for any $r,s\in\mc R$.
\end{proof}
As mentioned in \cite{Chakraborty2019}, subadditivity is not enough to show the stability of the game or the satisfaction of the members of a given coalition, \ie there is no incentive to disband coalitions. To do this, we utilize the following result on concavity.
\begin{theorem}\label{p4thm:concavity}
  Suppose Assumption~\ref{p4ass:coalcond} holds. The game with multiple energy retailers $(\mathcal{N},\nu)$ satisfies for any $r,s\in\mc{R}$,
  \begin{equation} \label{p4eq:concav}
    \nu(S_r \cup S_s)+\nu(S_r \cap S_s)\leq v(S_r)+v(S_s).
  \end{equation}
\end{theorem}
\begin{proof}
  As consequence of (\ref{p4eq:nosellers}) and (\ref{p4eq:nooverlap}) in Assumption~\ref{p4ass:coalcond}, $v(S_i \cup S_j)=0$ and $v(S_i \cap S_j)=v(\emptyset)=0$. From Theorem~\ref{p4thm:comps}, $v(S_r)\geq0$ holds for all $r\in\mc{R}$. Therefore \eqref{p4eq:concav} holds.
\end{proof}
A property of all $N$-player concave games is that they are balanced \cite{Rosen1965a}, meaning that for any concave game, there exists an equilibrium point. We can state the following result.
\begin{theorem}
The coalitional game with multiple energy retailers $(\mathcal{N},v )$ is balanced.
\end{theorem}
 \begin{proof}
 The proof is adapted from \cite[Theorem 1]{Rosen1965a}.
 \end{proof}
\begin{remark}
In both Theorem~\ref{thm:retailer_PC} and \ref{p4thm:concavity} similar inequalities are used to denote either convexity and concavity respectively. As mentioned in \cite{Granot1982}, the difference lies in the nature of the characteristic function used in the game: for the former $c(\cdot)$ is a cost function and measures how expensive is for players to form a coalition, whereas for the latter $\nu(\cdot)$ represents savings made by players joining a coalition.
\end{remark}
%
A consequence of Theorem~\ref{p4thm:comps} and Assumption~\ref{p4ass:coalcond} is that $\nu(\mc{N}) = 0$ since, by definition, the grand coalition contains all retailers. This fact renders conventional cooperative game methods, see \cite{Bausoa,Chakraborty2019}, unusable. To overcome this difficulty, we recur to notion of $\mbb{D}_{hp}$-stability introduced by \cite{Apt2006}.
%
\begin{definition}[$\mathbb{D}_{hp}$ stability]
  For $r\in\mc{R}$, the coalition $S_r\subset\mc N$ is $\mathbb{D}_{hp}$-stable if the following conditions are satisfied:
  \begin{enumerate}
  \item given a collection $\{P_{r_1},\ldots,P_{r_L}\}$ resulting from an arbitrary partition of $S_r$, such that $\cup_{j=1}^L P_{r_j}=S_r$:
    \begin{equation} \label{p4eq:cond1}
      \nu(S_r)\geq\sum_{j=1}^L \nu(P_{r_j}),~\forall i\in\mathcal{R},
    \end{equation}
  \item given a set $\mathcal{T}\subset\mc{R}$ with $|\mc T|\leq |\mc R|$:
    \begin{equation}\label{p4eq:cond2}
      \sum_{r\in \mathcal{T}}\nu(S_r)\geq \nu(\bigcup\limits_{r\in \mathcal{T}} S_r).
    \end{equation}
  \end{enumerate}
  \label{def:Dhp-stability}
\end{definition}
Our next result establishes the $\mbb{D}_{hp}$ stability of the coalitions formed in the game $( \mc{N},\nu )$.
\begin {theorem}
  Suppose Assumption \ref{p4ass:union} and \ref{p4ass:coalcond} hold. The retailer coalitions $S_r\subset\mc N$ are $\mathbb{D}_{hp}$ stable for the game with multiple retailers $( \mc{N},\nu )$.
\end{theorem}
\begin {proof}
  For a fixed $r\in\mathcal{R}$ and associated coalition $S_{r}\subset\mathcal{N}$, consider a collection $\{P_{r_1},\ldots,P_{r_L}\}$. It follows that $r\in P_{r_j}$ for some $j\in\{1,\ldots,L\}$ which implies $\nu(P_{r_j})\geq 0$, and $\nu(P_{i_k})=0$ for the rest. The cost associated with $P_{r_j}$ satisfies $c(S_r)\geq c(P_{r_j})$ since the MST of $P_{r_j}$ is contained in the one corresponding to $S_r$. Condition~\eqref{p4eq:cond1} follows directly from\[\nu(S_r) - \nu(P_{r_j})\geq c(S_r) - c(P_{r_j})\] since $c(\{r,b_{k}\}) > 0$ for all $b_k \in S_r\setminus P_{r_j}$.

  For the second condition,  \eqref{p4eq:nosellers} and \eqref{p4eq:nooverlap} imply that $\nu(\cup_{r\in \mathcal{T}} S_r)=0$. From the value formulation for a coalition~(\ref{p4eq:coalvalue}) and Assumption~\ref{p4ass:coalcond}, $\nu(S_i)\geq0$. From the above, condition~(\ref{p4eq:cond2}) holds.
\end{proof}
The coalitional game as formulated, following Algorithm~\ref{p4alg:alg1}, leads to a partition of the network in coalitions. The prices and demands from each retailer and consumer in such partition operate at a Stakelberg equilibrium. We refer the reader to our previous result in \cite[Theorem 2]{GenisMendoza2021}, where a similar market setup is presented together with the ways in which such an equilibrium is derived and characterised. We illustrate the existence and convergence towards a Stackelberg equilibrium for our game with the numerical results in Section \ref{p4sec:alltogether}. 

\section{Risk Sharing and Reduction of Statistical Dispersion}
\label{p4sec:stats}
In this section we provide a brief insight into the statistical implications for consumers when considering their demand to be a random variable. Motivated by \cite{Baeyens2013}, we consider the statistical properties of consumer demand in order to show that coalition formation entails lower risks. In the proposed game, the situation where a retailer coalition $S_r$ contains less consumers than expected constitutes a risk for consumers since subsidies are lower. Consumers can lower this risk by acting together in joining a retailer coalition which allows them to increase collective profit and sharing the risk.

For each $b\in \mc B$, consider $\alpha_b$ to be a random variable $\alpha_{{b}}(t)\sim N(\mu_b,\sigma^2)$. A direct consequence is that the power demand $P_b^d(t)\in [\underline{P}_b,\bar{P}_b]$ is a stochastic process; the vector demand is $P^d = (P_{1},\ldots P_{|\mc B|})\in \mc{P} = \prod_{b\in\mc B}[\underline{P}_b,\bar{P}_b]$. The cumulative distribution function (CDF), compactly supported on $\mc{P}\subset\Rset^{|\mc B|}$, at each time $t$ is given by
\begin{equation}
  \Phi(Q^d;t)=\mathbb{P}\{P^d(t)\leq Q^d\},
  \label{p4eq:prob}
\end{equation}
and denotes the probability of the random process $P^d(t)$ taking a value less than or equal to $Q^d$. The power consumption of a coalition $S_r\subseteq\mathcal{N}$ for $r\in\mc{R}$ would be then be represented by a sum of stochastic processes $  P^d_{S_r}(t)=\sum_{b\in S_r\cap\mc B}{P^d_{b}(t)}$. The corresponding random process is denoted by $\mathbf{P}^d_{S_r}=\{{P}^d_{S_r}(t)<P^d\}$ with support $\mc{P}_{S_r} = [\underline{P}_{S_r},\bar{P}_{S_r}]$ with $\underline{P}_{S_r} = \sum_{b\in S_r\cap\mc B}\underline{P}_b$ and similarly for $\bar{P}_{S_r}$. From this, the time averaged CDF is
\begin{equation}
  F_{S_r}(P^d)=\frac{1}{T}\int_{t_0}^{t_f}\Phi_{S_r}(P^d,t)dt.
  \label{p4eq:integral}
\end{equation}
The associated quantile function $F^{-1}_{S_r}:[0,1]\to\mc{P}_{S_r}$ is $F^{-1}_{S_r}(p)=\inf \{x\in[0,1]| F_{S_r}(x)\geq p\}$ for any $p\in [0,1]$. On the other hand, the total consumer profit for a given price $\Lambda_r$ and subsidies $\kappa_r\delta_b^r$ is $\Pi_{S_r}(P^d,r)=\sum_{b \in S_r\cap\mc B}{\Pi_{b}(P^d_{b},r)}$. The expected consumer profit is also a stochastic process and is given by:
\begin{equation}
  J_{S_r}(r)=\mathbb{E}~\Pi_{S_r}(\mathbf{P}^d_{S_r},r).
  \label{p4eq:expected}
\end{equation}
%
The individual profit of a consumer $b\in\mc{B}$ such that $S_r=\{r,b\}$ is denoted as $\Pi_{\{r,b\}}$. The following result shows that risk sharing through coalitions leads to an increase in profit \emph{almost surely}.\footnote{A property $P$ holds \emph{almost surely} if the points where its complement does not hold has measure zero.}.
\begin {proposition}
  Suppose Assumptions \ref{p4ass:union} and \ref{p4ass:coalcond} hold. For $r\in\mc R$, the following holds \emph{almost surely}, 
\begin{equation}\label{p4eq:almoststonchs}
\Pi_{S_r}(\mathbf{P}^d_{S_r},r)\geq\sum_{b \in S_r\cap \mc B}{\Pi_{\{r,b\}}(\mathbf{P}^d_{b_j},r)}.
\end{equation}
\end{proposition}
\begin{proof}
  To prove the result, we employ the properties of the sum of random variables \cite{Lemons2002}, homogeneity and superadditivity of functions of stochastic processes \cite{Baeyens2013}. Condition \eqref{p4eq:almoststonchs} is fulfilled as a direct consequence of Assumption~\ref{assum:savings}, where the value of a coalition is always greater or equal than the value of a single consumer with a retailer.  
\end{proof}
The above result establishes that coalitions always bring larger collective profits for consumers. These benefits can be directly attributed to the attenuation of statistical dispersion from aggregation; this phenomena has been also explored in \cite{Rajagopal} where the optimal expected profit directly depends on the deviation of the coalitional value-at-risk (CVaR) or \emph{coalitional shortfall deviation} \cite{Rockafellar2005}. For any $q\in(0,1)$, the CVaR deviation of $\mathbf{P}^d_{S_r}\sim F_{S_r}$ is defined as $\mathcal{D}_q(\mathbf{P}^d_{S_r}):=\mathbb{E}[\mathbf{P}^d_{S_r}]-\mathbb{E}[\mathbf{P}^d_{S_r}|\mathbf{P}^d_{S_r}\leq F_{S_r}^{-1}(q)]$. This deviation measures the difference between the expected value and the probability of $q$ being near the bounds of the probability distribution. Then, the reduction of dispersion that is induced by consumers joining a coalition $S_r$ and aggregating their demand is $\mb{\Delta}_{S_r}:=\sum_{b\in S_r}\mc{D}_q(\mathbf{P}^d_{b})-\mc{D}_q(\mathbf{P}^d_{S_r})$. The non-negativity of $\mathbf{\Delta}_{S_i}$ for all $S_r\subset \mc N$ follows immediately. This establishes that CVaR is a measure of the reduction in statistical dispersion. Moreover, there is an inversely proportional relation between expected profits and statistical dispersion of aggregate consumption.
\section{Implementation with Physical Dynamics}
\label{p4sec:alltogether}
In this section, we analyse the effect of the proposed pricing scheme on the voltages of a low voltage and resistive micro-grid, see Fig. \ref{p4fig:resdorf}. An element of the adjacency matrix $A_{ij}$ corresponds to a conductance $1/R_{ij}$ between nodes~$i$ and~$j$ and the network conductance matrix  $G = \mc{L} + R^{-1}$ where $R = \ts{diag}(\{R_{ii}\}_{i\in\mc{N}})$ is the impedance of each node \cite{Dorfler}.
\subsection{Micro-Grid and Demand Dynamics}
\label{p4sec:dynamics}
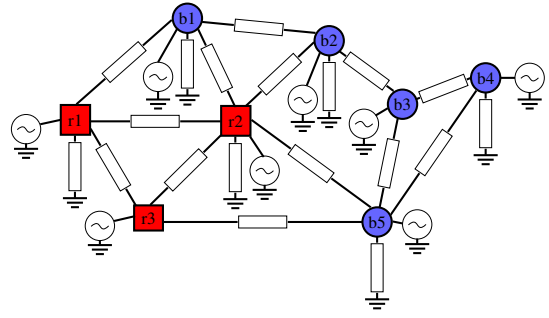
\begin{figure}
\centering
\scalebox{.65}{
\begin{tikzpicture}[circuit ee IEC]
\draw[very thick] (-1.62,-0.425) node [ground, rotate=-90] {};
\draw[very thick](-1.62,-0.325)to[resistor](-1.62,0.95);
\draw[very thick,fill=blue!60] (-1.62,1.25) circle (.3) node[anchor=center] {b5};
\draw[very thick](-0.8,0.805)to(-0.8,0.91) ;
\draw[very thick] (-0.8,0.705) node [ground, rotate=-90] {};
\draw (-0.8,1.2) circle (.3cm); 
\draw [very thick]  (-1.3,1.26) -- (-1.1,1.22);
\draw (-0.98,1.19) .. controls (-0.75,1.39) and (-0.84,0.99) .. (-0.62,1.2);
\draw[very thick](1.5,3.765)to(1.5,3.91) ;
\draw[very thick] (1.5,3.665) node [ground, rotate=-90] {};
\draw (1.5,4.2) circle (.3cm); 
\draw [very thick]  (0.9,4.2) -- (1.2,4.2);
\draw (1.32,4.19) .. controls (1.55,4.39) and (1.46,3.99) .. (1.68,4.2);
\draw[very thick] (0.6,2.525) node [ground, rotate=-90] {};
\draw[very thick](0.6,2.625)to[resistor](0.6,3.9);
\draw[very thick,fill=blue!60] (0.6,4.2) circle (.3) node[anchor=center] {b4};
\draw[very thick] (-2.6,3.285) node [ground, rotate=-90] {};
\draw[very thick](-2.6,3.385)to[resistor](-2.6,4.66);
\draw[very thick,fill=blue!60] (-2.6,4.96) circle (.3) node[anchor=center] {b2};
\draw[very thick](-3.14,3.305)to(-3.14,3.46) ;
\draw[very thick] (-3.14,3.205) node [ground, rotate=-90] {};
\draw (-3.14,3.75) circle (.3cm); 
\draw [very thick]  (-2.79,4.71) -- (-3.06,4.04);
\draw (-3.32,3.74) .. controls (-3.09,3.94) and (-3.18,3.54) .. (-2.96,3.75);
\draw[very thick] (-5.5,3.745) node [ground, rotate=-90] {};
\draw[very thick](-5.5,3.845)to[resistor](-5.5,5.12);
\draw[very thick,fill=blue!60] (-5.5,5.42) circle (.3) node[anchor=center] {b1};
\draw[very thick](-6.1,3.765)to(-6.1,3.93) ;
\draw[very thick] (-6.1,3.665) node [ground, rotate=-90] {};
\draw (-6.1,4.22) circle (.3cm); 
\draw [very thick]  (-5.7,5.22) -- (-6.1,4.52);
\draw (-6.28,4.21) .. controls (-6.05,4.41) and (-6.14,4.01) .. (-5.92,4.22);
\draw[very thick,fill=blue!60] (-1.1,3.66) circle (.3) node[anchor=center] {b3};
\draw[very thick](-1.89,2.835)to(-1.89,2.97) ;
\draw[very thick] (-1.89,2.735) node [ground, rotate=-90] {};
\draw (-1.89,3.26) circle (.3cm); 
\draw [very thick]  (-1.4,3.58) -- (-1.65,3.43);
\draw (-2.07,3.25) .. controls (-1.84,3.45) and (-1.93,3.05) .. (-1.71,3.26);
\draw[very thick](-8.8,2.755)to(-8.8,2.86) ;
\draw[very thick] (-8.8,2.655) node [ground, rotate=-90] {};
\draw (-8.8,3.15) circle (.3cm); 
\draw [very thick]  (-8.05,3.35) -- (-8.51,3.25);
\draw (-8.98,3.14) .. controls (-8.75,3.34) and (-8.84,2.94) .. (-8.62,3.15);
\draw[very thick] (-7.8,1.645) node [ground, rotate=-90] {};
\draw[very thick](-7.8,1.745)to[resistor](-7.8,3.02);
\draw[very thick](-3.92,1.915)to(-3.92,2.01) ;
\draw[very thick] (-3.92,1.815) node [ground, rotate=-90] {};
\draw (-3.92,2.3) circle (.3cm); 
\draw [very thick]  (-4.22,3) -- (-3.94,2.58);
\draw (-4.1,2.29) .. controls (-3.87,2.49) and (-3.96,2.09) .. (-3.74,2.3);
\draw[very thick,fill=red] (-7.49,3.645) rectangle (-8.09,3.035) node[pos=.5]  {r1};
\draw[very thick] (-4.52,1.625) node [ground, rotate=-90] {};
\draw[very thick](-4.52,1.725)to[resistor](-4.52,3);
\draw[very thick,fill=red] (-4.21,3.625) rectangle (-4.81,3.015) node[pos=.5]  {r2};
\draw[very thick,fill=red] (-6,1.59) rectangle (-6.59,1.075) node[pos=.5]  {r3};
\draw[very thick](-7.29,0.785)to(-7.29,0.88) ;
\draw[very thick] (-7.29,0.685) node [ground, rotate=-90] {};
\draw (-7.29,1.17) circle (.3cm); 
\draw [very thick]  (-6.57,1.37) -- (-7,1.27);
\draw (-7.47,1.16) .. controls (-7.24,1.36) and (-7.33,0.96) .. (-7.11,1.17);
\draw[very thick](-1.92,1.25)to[resistor](-6,1.25);
\draw[very thick](-4.82,3.31)to[resistor](-7.5,3.31);
\draw[very thick](-2.88,5.12)to[resistor](-5.2,5.345);
\draw[very thick](-5.82,5.385)to[resistor](-7.77,3.66);
\draw[very thick](-5.28,5.195)to[resistor](-4.52,3.6);
\draw[very thick](-2.91,4.93)to[resistor](-4.31,3.62);
\draw[very thick](-2.4,4.75)to[resistor](-1.28,3.93);
\draw[very thick](0.3,4.2)to[resistor](-0.83,3.76);
\draw[very thick](0.41,3.94)to[resistor](-1.34,1.4);
\draw[very thick](-1.2,3.36)to[resistor](-1.57,1.54);
\draw[very thick](-4.19,3.34)to[resistor](-1.75,1.55);
\draw[very thick](-4.81,3.03)to[resistor](-6.26,1.6);
\draw[very thick](-7.48,3.16)to[resistor](-6.55,1.58);
\end{tikzpicture}
}
\caption{Resistive micro-grid in a network representation, comprised by loads/consumers \protect\marksymbol{*}{blue!60} and generators/retailers \protect\marksymbol{square*}{red} , resistive distribution lines, and shunt conductances.}
\label{p4fig:resdorf}
\end{figure}
We consider islanded micro-grids since these allow to reframe the application within the proposed multiple retail scenario\footnote{The grid connected case can be addressed by adding an extra retailer node with additional properties such as wider generation capabilities.}. We consider a network where both retailers and consumers provide grid support by operation under $P\sim V$ droop control. The droop controller used in our setting is similar in functionality to the traditional one \cite{Wang}, albeit this version uses bounded integrators \cite{Konstantopoulos2016} which are capable of voltage constraint satisfaction. The dynamics of each node $i\in\mc N$ are
\begin{equation}
  \begin{split}
    \tau_{v,i}\dot V_i & =   \big({V^*} - V_i  + \\
    & - {k_i} \sum_{j\in\mc{N}}G_{ij}V_iV_j + k_i {P_i^\ts{set}}\big) \bigg(1-\frac{(V_i-V^*)^2}{\Delta V_i^2}\bigg),
  \end{split}
\label{p4eq:droopy2}
\end{equation}
where $\tau_{v,i},k_i>0$ are the time constant and droop coefficient, $V^*\in\Rset$ is the rated voltage, and $\Delta V_i>0$ define desired voltage bounds. The following result shows that voltages remain within prespecified bounds.
\begin{lemma}
  If $V_i (0) \in(V^* -\Delta V,V^* +\Delta V)$, $\forall i \in \mathcal{N}$, then $\forall t\geq 0$ and $\forall i \in \mathcal{N}$, $V_i (t) \in[V^* -\Delta V,V^* +\Delta V]$.
  \label{lem:constraints}
\end{lemma}
\begin{proof}
  The proof follows from contradiction. Suppose there exists a voltage violating the constraints. By continuity of solutions, there exist $i\in\mc{N}$ and $\hat{t}>0$ such that $|V_i(\hat{t})-V^*| = \Delta V$. This point would be an equilibrium of \eqref{p4eq:droopy2}, \ie $\dot{V}_i = 0$. This implies that the voltage trajectory $V_i(t)$ cannot leave the set $(V^* -\Delta V,V^* +\Delta V)$. This proves our result.
\end{proof}
The vector $P^{set}$ contains the power reference (demanded or consumed) for each node. For consumers, the $P_i^\ts{set}$ contains the first order demand response \cite{GenisMendoza2019,pricingDR} for all $b\in\mc{B}$
\begin{equation}
  \tau_i \dot P_{b}^\ts{set} =P_{b}^d- P_{b}^\ts{set},
  \label{p4eq:fodyn}
\end{equation}
where $\tau_i >0$ is the associated time constant and the input $P_b^d$ is the output of \eqref{p4eq:optcons_mi}. The overall state vector is $x = \{V_{b}, V_{r},P_{b}^\ts{set}\}_{b \in \mathcal{B},r \in \mathcal{R}}$.
\subsection{Physical Stability Analysis}
In this section, we analyse the stability of the micro-grid given by \eqref{p4eq:droopy2} and \eqref{p4eq:fodyn}. To ensure that our problem is well posed, we require the following assumption:
\begin{assumption} \label{p4ass:eqprelax}
  For constant inputs $ P_r ^\ts{set}\in\Rset$ for all $r\in\mc{R}$ and $P_b^d$ for all $b\in\mc{B}$, there exists an equilibrium point $\{\bar V_{{b}},\bar V_{{r}},\bar P_{{b}}^\ts{set}\}$ $\forall  b \in \mathcal{B},r \in \mathcal{R}$ for the coupled system \eqref{p4eq:droopy2}-\eqref{p4eq:fodyn} such that $|V_i-V^*|\leq \Delta V $ for all $i\in\mc{N}$.
\end{assumption}
If Assumption~\ref{p4ass:eqprelax} does not hold, the equilibrium voltage would lie, following Lemma~\ref{lem:constraints}, on the boundary of the constraint sets. 
\begin{proposition}
  Suppose Assumption \ref{p4ass:eqprelax} holds. The solutions of \eqref{p4eq:droopy2}-\eqref{p4eq:fodyn} satisfy $\lim_{t\to\infty}|x-\bar{x}| = 0$  with $\bar{x} = \{\bar V_{{b}},\bar V_{{r}},\bar P_{{b}}^\ts{set}\}_{b\in\mc{B},r\in\mc{R}}$ and $x(0)\in\mc{A}$ a neighbourhood of $\bar{x}$ if for all $i\in\mc{N}$
  \begin{equation}
    k_i\Delta V \sum_{j\in\mc V}\frac{1}{R_{ij}} < \frac{1}{2} + k_i\frac{V^*}{R_{ii}}.
    \label{p4eq:conddiscq}
  \end{equation}
\end{proposition}
\begin{proof}
  The Jacobian of \eqref{p4eq:droopy2}-\eqref{p4eq:fodyn} at  $\bar{x}\in\Rset^{|\mc N|+|\mc B|}$ is:
\begin{equation}\label{p4eq:jacobianBIC}
  J(\bar{x}) =\begin{bmatrix}\Upsilon&-\kappa\\
  \mathbf{0} &  -\tau^{-1}
  \end{bmatrix},
\end{equation}
where $\Upsilon = \frac{\partial f_v}{\partial V}\bigl\vert_{\bar{x}}\in\Rset^{|\mc N|\times|\mc N |}$, $\kappa=\frac{\partial f_v}{\partial P^\ts{set}}\bigl\vert_{\bar{x}}\in\Rset^{|\mc N|\times |\mc B|}$, and $\tau\in\Rset^{|\mc B|\times |\mc B|}$ is a diagonal matrix of time constants and $f_v(\cdot,\cdot)$ is the voltage vector field of \eqref{p4eq:droopy2}. The required stability condition follows from the eigenvalues of $J(\bar{x})$; negative eigenvalues guarantee a stable system at $\bar{x}$. We exploit the structure of $J(\bar{x})$ to compute its eigenvalues\footnote{$\mbb{I}_n$ denotes the identity in $\Rset^n$.},
\[\det(\eta\mbb{I}_{|\mc N|+|\mc B|} - J(\bar{x})) = \det(\eta\mbb{I}_{|\mc N|} - \Upsilon)\det(\eta\mbb{I}_{|\mc B|}+\tau^{-1}).\]
Clearly, $-\tau^{-1}$ contributes only negative eigenvalues, we, therefore, only analyse the behaviour of the first factor. We note that the quantity $\xi = (1- (V^*-V_i)^2/\Delta V_i^2)$ is positive by Assumption~\ref{p4ass:eqprelax} for all $i\in\mc{N}$, and $\Upsilon$ can be factored as $\Upsilon = \xi \tilde{\Upsilon}$ where $\tilde \Upsilon_{ii} = c_i\tau_{v,i}^{-1}(-1 - k_i(\sum_{j\neq i}G_{ij}V_j + 2G_{ii}V_i))$ and $\tilde\Upsilon_{ij} = -c_i\tau_{v,i}^{-1}k_iG_{ij}V_i$. As a result, the determinant we seek is: $\det(\eta\mbb{I}_{|\mc N|} - \tilde\Upsilon)$. We use Gershgorin discs $\bm{\Delta}_i(C_i,R_i)$ with centre at $C_i=\tilde \Upsilon_{ii}$ and radius $\varepsilon_i=\sum_{j\neq i}|{\tilde{\Upsilon}_{ij}}|$ to analyse the eigenvalue location. If condition~\eqref{p4eq:conddiscq} holds, the eigenvalues of $\ts{diag}(\{\tilde\Upsilon_{ii}\}_{i\in\mc{N}})$ lie on the left side of the complex plane. The equilibrium $\bar{x}$ is asymptotically stable in a neighbourhood $\mc{A}\subset\Rset^{|\mc{N}|+|\mc{B}|}$ where the linearisation holds.
\end{proof}
The above results allow us to check the stability of the micro-grid based on \eqref{p4eq:conddiscq} without evaluating directly the equilibrium point which may be difficult to compute; a computation of the equilibrium point requires a solution of the power flow equations. Furthermore, using a continuity argument, our result holds for equilibrium points on the boundary of $(V^*-\Delta V_i,V^*+\Delta V_i)$. 
\section{Simulations}\label{p4sec:simulations}
 \subsection{Coalition Formation and Profit Calculation}\label{p4sec:sim1}
 To illustrate our scheme, we have formulated two scenarios. The first one consists of a micro-grid with a single retailer, whereas the second one considers two additional retailers; both scenarios have five consumers $\mc{B} = \{b_1,\ldots,b_5\}$. The parameters for both scenarios are listed in Table \ref{p4tab:param} and their cost networks of different retailers are shown in Fig. \ref{p4fig:CostNetworks}. The games are played periodically, we use a period of $10$ seconds for both scenarios.
\begin{table}
  \caption{Parameters for Retailers and Consumers.\label{p4tab:param}}
  \centering
  \begin{tabular}{cccccc}
    \hline 
\textbf{Retailer} & $ \alpha_{{r}_i}$ & $\kappa_{{r}_i}$ & $ \underline\lambda_{i}$ & $\overline \lambda_i $& $P^g_{i}$\\ \hline
$r_1$             & 1e-4 $\$^\frac{1}{2}$               & 65 \$                   & 0.01 \$/W                         & 4 \$/W                       & 30 kW                   \\
$r_2$             & 7e-5 $\$^\frac{1}{2}$                & 64 \$                 & 0.01 \$/W                         & 2 \$/W                       & 30 kW                    \\
$r_3$             & 5e-5 $\$^\frac{1}{2}$                & 63 \$                & 0.01 \$/W                         & 3.5 \$/W                     & 30 kW                    \\ \hline
&&&&\\
\end{tabular}
\begin{tabular}{ccccc}
\hline
\textbf{Consumer} & $ \alpha_{{b}_j}$ & $P^{d_{rated}}_{{b}_j}$& $\underline \zeta_{{b}_j}$ & $\overline \zeta_{{b}_j}$  \\ \hline
$b_1$             & 1800 W$^6$               & 3 kW& 0 kW                           & 6 kW                                        \\
$b_2$             & 150 W$^6$                  & 3.5 kW& 0 kW                            & 7 kW                                         \\
$b_3$             & 140 W$^6$                 & 2.8 kW& 0  kW                           & 5.6 kW                                       \\
$b_4$             & 100 W$^6$                 & 4 kW& 0 kW                            & 8 kW                                         \\
$b_5$             & 1600 W$^6$               & 1.5 kW& 0 kW                             & 3 kW                                       \\ \hline
\end{tabular}
\end{table}
\begin{figure}[t!]
\centering
\vspace{-2mm}
\begin{tikzpicture}[scale=0.8]
\draw[ultra thick] (-1.6,18.15) circle (.3) node[anchor=center] {$b_3$};	   
\draw[ultra thick] (-0.2,18.5) circle (.3) node[anchor=center] {$b_4$};	   
\draw[ultra thick] (0.7,19.65) circle (.3) node[anchor=center] {$b_5$};	
\draw[ultra thick](-1.6,21.15) circle (.3) node[anchor=center] {$b_1$};
\draw[ultra thick] (-0.15,20.7) circle (.3) node[anchor=center] {$b_2$};
\draw[blue, ultra thick] (-3.5,19.65) circle (.3) node[anchor=center] {$r_1$};
\draw[blue] (-2.6764,21) node[anchor=center] {550};
\draw[blue](-0.95,18.95) node[anchor=center] {475};
\draw[blue] (-0.5,19.75) node[anchor=center] {535};
\draw[blue](-1.9088,19.8588) node[anchor=center] {300};
\draw[blue](0.8176,20.4) node[anchor=center] {280};
\draw[blue](0.8,18.85) node[anchor=center] {245};
\draw[blue](-0.95,20.7) node[anchor=center] {540};
\draw[blue](-2.9,18.5) node[anchor=center] {455};
\draw[blue, ultra thick](-3.25,19.8)to (-0.45,20.6);
\draw[blue, ultra thick](-3.3,19.4)to (-0.5,18.55);
\draw[blue, ultra thick](-1.9,21.15)to (-3.5,19.95);
\draw[blue, ultra thick](-1.6,18.45)to (-1.6,20.85);
\draw[blue, ultra thick](0.7,19.95)to (0.15,20.65);
\draw[blue, ultra thick](0.1,18.45)to (0.65,19.35);
\draw[blue, ultra thick](-3.5,19.35)to (-1.9,18.15);
\draw[blue, ultra thick](0.4,19.55)to (-3.2,19.55);
\draw[ultra thick] (2.5,16.05) circle (.3) node[anchor=center] {$b_3$};	   
\draw[ultra thick] (3.6,16.4) circle (.3) node[anchor=center] {$b_4$};	   
\draw[ultra thick] (4.9,17.55) circle (.3) node[anchor=center] {$b_5$};	
\draw[ultra thick](2.6,19.05) circle (.3) node[anchor=center] {$b_1$};
\draw[ultra thick] (3.55,18) circle (.3) node[anchor=center] {$b_2$};
\draw[magenta, ultra thick] (0.7,17.55) circle (.3) node[anchor=center] {$r_3$};
\draw[magenta] (1.6,18.9) node[anchor=center] {460};
\draw[magenta](2.3,17.1) node[anchor=center] {495};
\draw[magenta] (3.1,17.25) node[anchor=center] {525};
\draw[magenta](3.9176,17.05) node[anchor=center] {230};
\draw[magenta](4.1,18.7) node[anchor=center] {400};
\draw[magenta](3.3264,15.8588) node[anchor=center] {340};
\draw[magenta](2.2,18) node[anchor=center] {540};
\draw[magenta](2.7236,18.5) node[anchor=center] {250};
\draw[magenta](1.62,16.1) node[anchor=center] {500};
\draw[magenta, ultra thick](0.95,17.7)to (3.25,17.9);
\draw[magenta, ultra thick](0.9,17.3)to (3.3,16.45);
\draw[magenta, ultra thick](2.3,19.05)to (0.7,17.85);
\draw[magenta, ultra thick](3.55,16.7)to (3.55,17.7);
\draw[magenta, ultra thick](3.5,18.3)to (2.85,18.85);
\draw[magenta, ultra thick](4.8,17.85)to (2.85,19.25);
\draw[magenta, ultra thick](2.8,15.95)to (3.4,16.2);
\draw[magenta, ultra thick](0.7,17.25)to (2.2,16.05);
\draw[magenta, ultra thick](4.6,17.45)to (1,17.45);
\draw[ultra thick] (-1.6,14.15) circle (.3) node[anchor=center] {$b_3$};	   
\draw[ultra thick] (-0.2,14.5) circle (.3) node[anchor=center] {$b_4$};	   
\draw[ultra thick] (0.7,15.65) circle (.3) node[anchor=center] {$b_5$};	
\draw[ultra thick](-1.6,17.15) circle (.3) node[anchor=center] {$b_1$};
\draw[ultra thick] (-0.15,16.7) circle (.3) node[anchor=center] {$b_2$};
\draw[black!60!green, ultra thick] (-3.5,15.65) circle (.3) node[anchor=center] {$r_2$};
\draw [black!60!green] (-2.65,17) node[anchor=center] {505};
\draw [black!60!green](-1.6412,15.15) node[anchor=center] {520};
\draw [black!60!green] (-1.65,15.75) node[anchor=center] {515};
\draw  [black!60!green](-0.9,16.9) node[anchor=center] {425};
\draw  [black!60!green](0.8,16.4) node[anchor=center] {225};
\draw  [black!60!green](0.8264,14.9) node[anchor=center] {350};
\draw  [black!60!green](-0.05,16) node[anchor=center] {365};
\draw  [black!60!green](-2.3,16.3) node[anchor=center] {510};
\draw  [black!60!green](-2.9,14.5) node[anchor=center] {525};
\draw[black!60!green, ultra thick](-3.25,15.8)to (-0.45,16.6);
\draw[black!60!green, ultra thick](-3.3,15.4)to (-0.5,14.55);
\draw[black!60!green, ultra thick](-1.9,17.15)to (-3.5,15.95);
\draw[black!60!green, ultra thick](-0.35,14.8)to (-1.4,16.9);
\draw[black!60!green, ultra thick](0.7,15.95)to (0.15,16.65);
\draw[black!60!green, ultra thick](0.1,14.45)to (0.65,15.35);
\draw[black!60!green, ultra thick](-1.45,14.4)to (-0.25,16.4);
\draw[black!60!green, ultra thick](-3.5,15.35)to (-1.9,14.15);
\draw[black!60!green, ultra thick](0.4,15.55)to (-3.2,15.55);
\end{tikzpicture}
\vspace{-3mm}
\caption{Cost networks defined by each retailer for the same set of consumers.}\label{p4fig:CostNetworks}
\vspace{-3mm}
\end{figure}
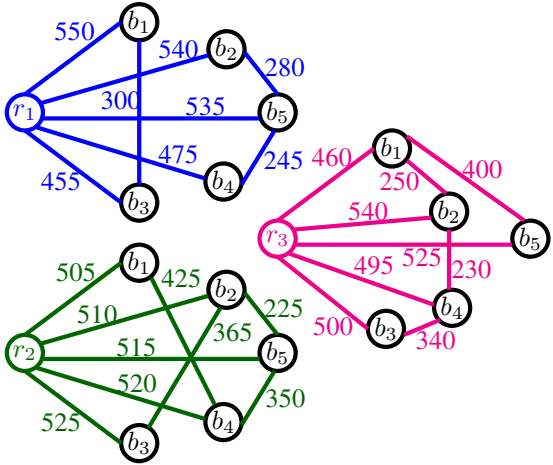
Figure \ref{p4fig:PricePower} shows the rationality of players, \ie consumers tend to consume more (less) given a lower (higher) price, and that retailers tend to lower (raise) their price when consumption is low (high). This is more evident in the single retailer scenario where both prices and demands settle. In addition, Fig. \ref{p4fig:PricePower} shows that some coalitions choose not to consume, this behaviour is the result of consumers choosing a coalition maximising their profit and leaving some retailers without consumers. 
\begin{figure}[t!]
\centering
\includegraphics[scale=0.65]{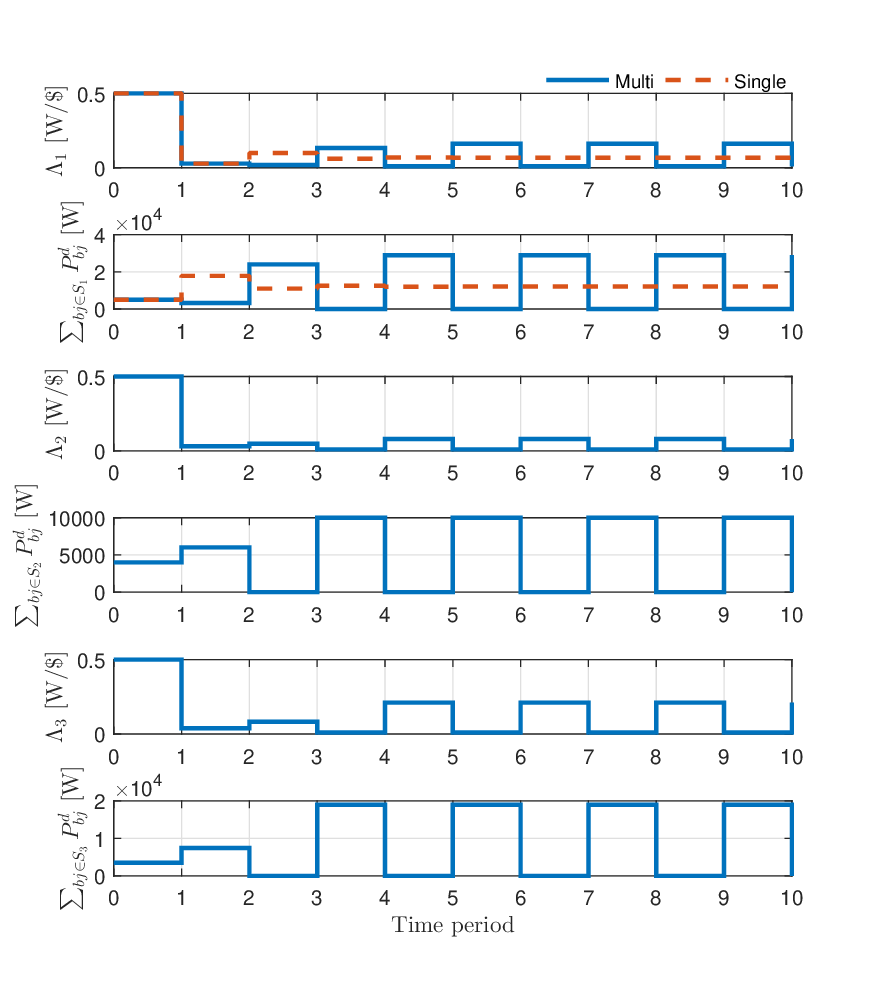}
\vspace{-1.4cm}
\caption{Price and coalition consumption responses over time.}\label{p4fig:PricePower}
\end{figure}
Figure \ref{p4fig:IndConsumptions} shows individual power demand for all $b\in \mc B$; each consumer is able to consume more power in the multiple retailer scenario. Each consumer power demand settles above their rated consumption even for those who do not prioritise consumption, \ie $\alpha_b$ is small. The profit obtained by each consumer, see Fig.~\ref{p4fig:BuyerProfits}, shows that greater profits can be obtained in the latter scenario. The case of consumer $b_4$ is of particular interest: this buyer is forced to pay a connection fee in the single retailer scenario; whereas it has no longer to pay such fee when there are more retailers.
\begin{figure}[t!]
\centering
\vspace{-2mm}
\includegraphics[scale=0.65]{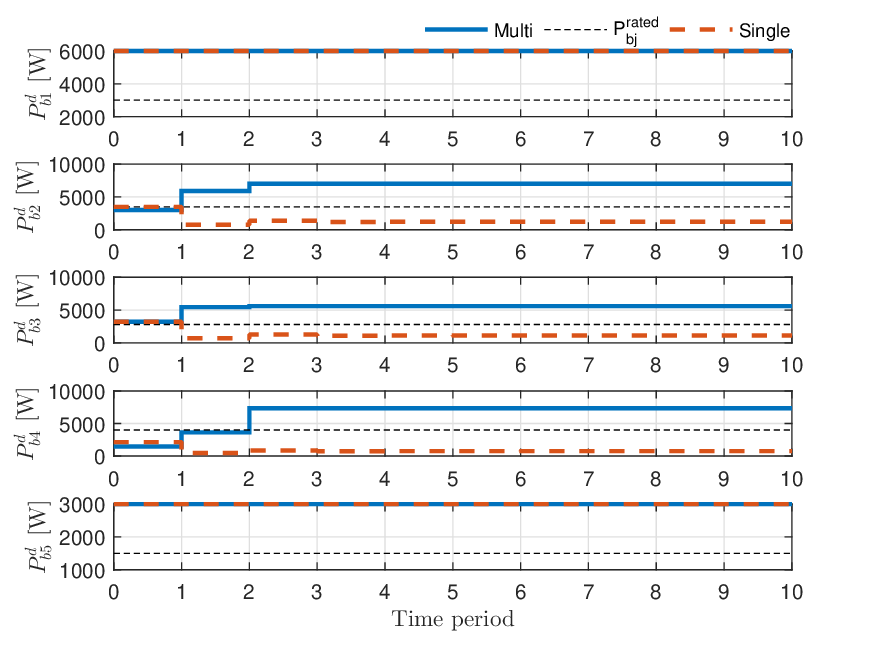}
\vspace{-11mm}
\caption{Consumers' individual power demand over time.}\label{p4fig:IndConsumptions}
\end{figure}
\begin{figure}
\centering
\includegraphics[scale=0.65]{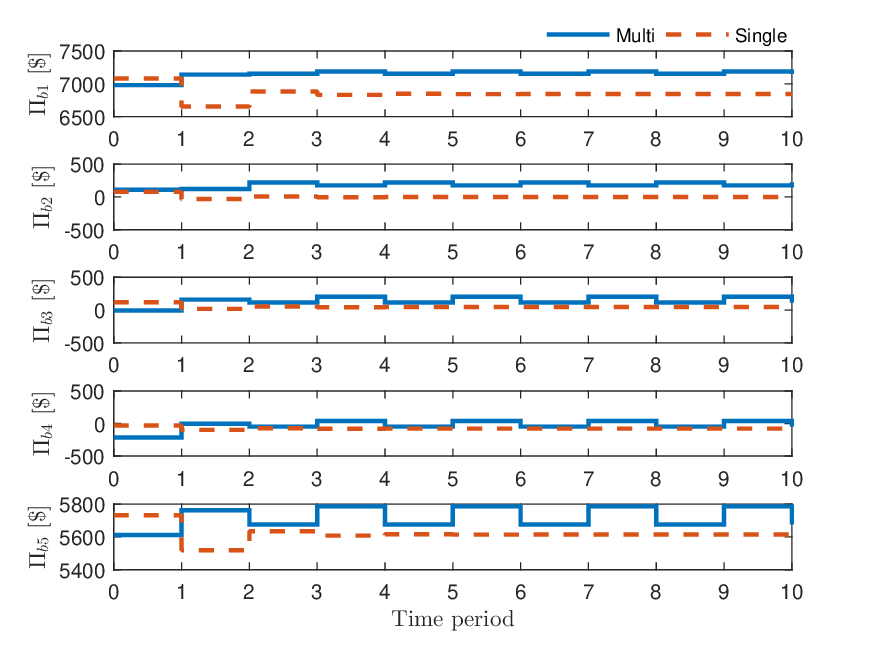}
\vspace{-11mm}
\caption{Consumers' immediate profits for every time period.}\label{p4fig:BuyerProfits}
\end{figure}
The profits of retailers, see Fig. \ref{p4fig:SellerProfits}, also change from scenario to scenario; as expected the profits of $r_1$ decrease in case of competition. When retailers fail to attract consumers, their profits can become negative as seen for $r_2$. The benefits of a multiple retailer scenario are clear from Figs.~\ref{p4fig:BuyerProfits} and~\ref{p4fig:SellerProfits}, consumers gain more from retailer competition. Lastly, Fig.~\ref{p4fig:CoalitionEvo} shows the evolution of the choice of coalition for each consumer. 
\begin{figure}[t!]
\centering
\vspace{-4mm}
\includegraphics[scale=0.65]{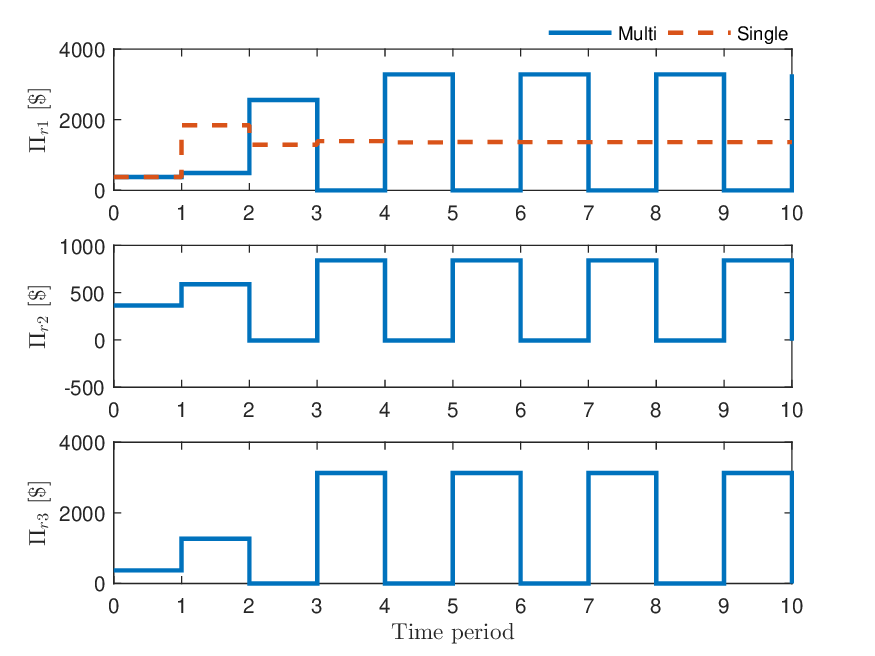}
\vspace{-9mm}
\caption{Retailers' immediate profits for every time period.}\label{p4fig:SellerProfits}
\end{figure}
%
%
\begin{figure}[t!]
\centering
\includegraphics[scale=0.65]{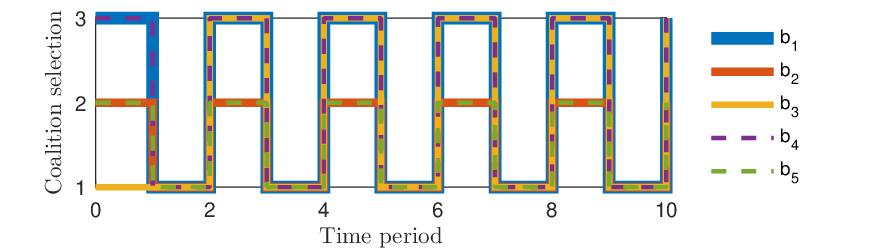}
\vspace{-8mm}
\caption{Coalition selection over time}\label{p4fig:CoalitionEvo}
\end{figure}
\subsection{Game and Physical Dynamics Integration}
In this section, we show the response of the physical system to the changing set-points arising from the proposed pricing scheme. The parameters are the following: droop control time constant $\tau_{vi}=0.1$s, load response time constant $\tau_i=3$s, nominal voltage $V^*=220$V, voltage deviation $\Delta V=11$V, droop coefficients $k_i=0.05V_i^*/P_{i}^{rated}$ and $c_i=(\pi \Delta V)/(0.1k_i P_{i}^{rated})$, where the $P_i^{rated}$ vector consists of the $P_{{r}}^g$ and $P_{{b}}^{Lrated}$ values for the corresponding retailer and consumers respectively. The coalition selection game is played every $T_S=60$s.
\begin{figure}[t!]
\centering
\includegraphics[scale=0.65]{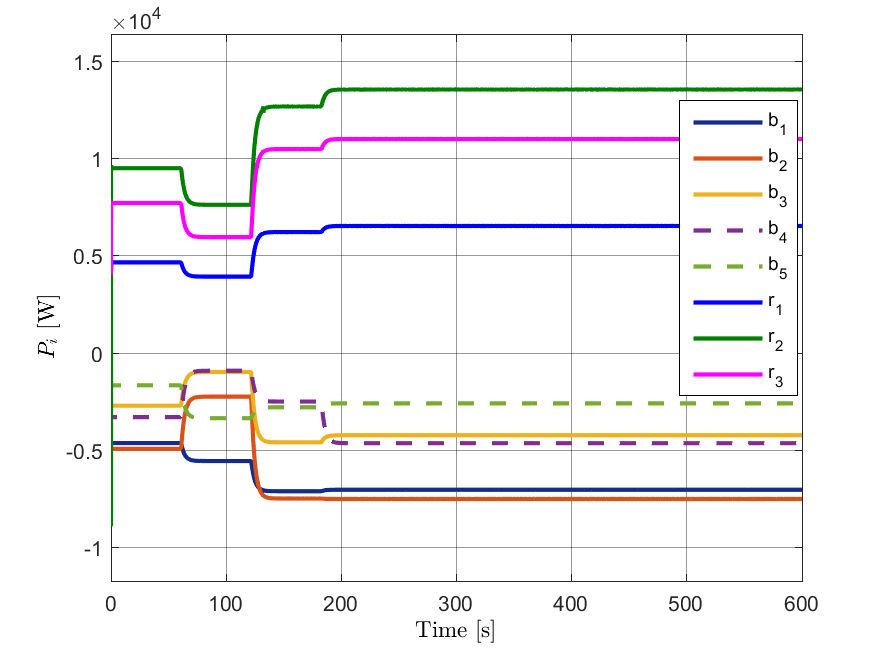}
\vspace{-8mm}
\caption{Resistive network node powers when subject to the proposed coalitional game.}\label{p4fig:physpowers}
\end{figure}
\begin{figure}[t!]
\centering
\includegraphics[scale=0.65]{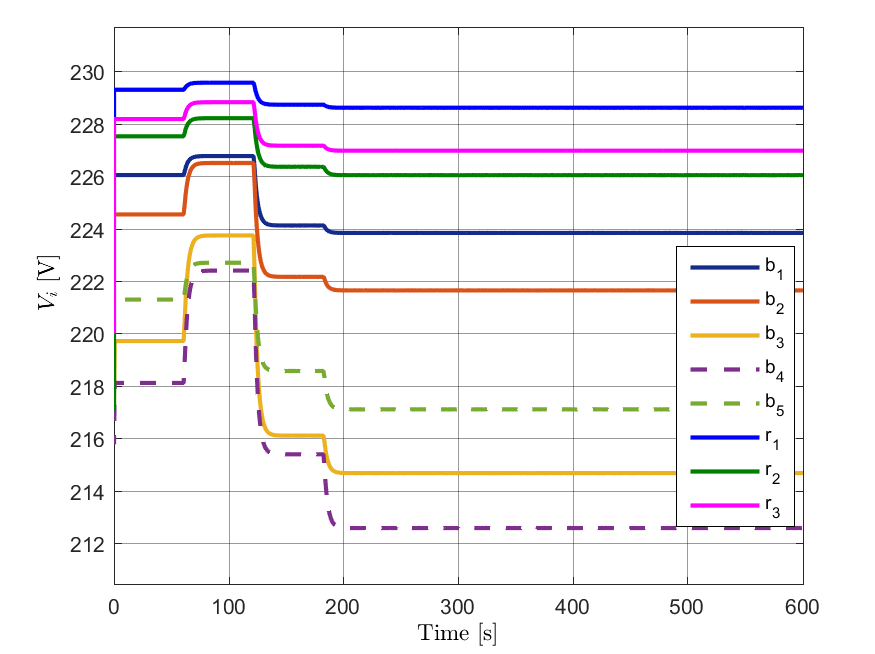}
\vspace{-8mm}
\caption{Resistive network node voltages when subject to the proposed coalitional game.}\label{p4fig:voltages}
\end{figure}
%
%
The response of the physical system to changes in power set-points, see Fig.~\ref{p4fig:physpowers} and \ref{p4fig:voltages}, for both retailers and consumers is seen through the voltage changes (voltage deviations allow for current to flow between nodes). We show that all voltages are bounded around the nominal voltage, and the power losses are small comparing Fig. \ref{p4fig:physpowers} with \ref{p4fig:IndConsumptions}. 
 %
 \section{Conclusions and Future Directions}
 \label{p4sec:theend}
In this paper we have proposed an on-line pricing scheme that encompasses the concepts of a hierarchical structure with the Stackelberg game and coalitional games with competing players. The definitions and the algorithm for the game have been established and the stability of the game has been shown. We have described a potential implementation of the present scheme integrated with the physical micro-grid dynamics, together with its corresponding stability analysis. A comparison between single and multiple retailer scenarios has been shown numerically, demonstrating the advantages of the latter from an economic point of view. Future directions of this work dwell in the incorporation of disconnection penalties, quality of the service, among other factors that can be included into the algorithm. Additionally, several approaches/variations for deriving the proposed cost networks can be investigated. Finally, with the same market setup and dynamics as basis, a plethora of optimisation methods can also be implemented for the design of novel dynamic pricing algorithms.
\bibliography{bibliography}
\bibliographystyle{IEEEtran}
\end{document}